\documentclass[11pt,reqno]{amsart}

\usepackage{amssymb, amsmath, amsthm}
\usepackage{hyperref}
\usepackage[alphabetic,lite]{amsrefs}
\usepackage{dlfltxbcodetips}
\usepackage{verbatim}
\usepackage{amscd}   % for commutative diagrams
\usepackage[all]{xy} % for complicated commutative diagrams
\usepackage{youngtab} % for Young tableaux
\usepackage{young} % for Young tableaux
\usepackage{tikz}

\newcommand{\defi}[1]{{\upshape\sffamily #1}}

\renewcommand{\a}{\alpha}

\renewcommand{\b}{\beta}

\renewcommand{\c}{\gamma}

\renewcommand{\ll}{\lambda}

\renewcommand{\o}{\otimes}

\newcommand{\s}{\sigma}

\newcommand{\bb}[1]{\mathbb{#1}}
\renewcommand{\rm}[1]{\mathrm{#1}}
\newcommand{\mc}[1]{\mathcal{#1}}

\newcommand{\ccirc}[1]{\xymatrix@1{*+<1ex>[o][F-]{#1}}}

\def\ccone{\ccirc{1}}
\def\cctwo{\ccirc{2}}
\def\ccthree{\ccirc{3}}

\def\lra{\longrightarrow}

\newtheorem{theorem}{Theorem}[section]
\newtheorem{lemma}[theorem]{Lemma}
\newtheorem{conjecture}[theorem]{Conjecture}
\newtheorem{proposition}[theorem]{Proposition}
\newtheorem{corollary}[theorem]{Corollary}
\newtheorem{question}[theorem]{Question}

\newtheorem*{tmaincor}{Corollary \ref{thmmaincor}}
\newtheorem*{tmain}{Theorem \ref{thmmain}}

\theoremstyle{definition}
\newtheorem{definition}[theorem]{Definition}
\newtheorem*{definition*}{Definition}
\newtheorem{example}[theorem]{Example}

\theoremstyle{remark}
\newtheorem{remark}[theorem]{Remark}
\newtheorem*{remark*}{Remark}

\numberwithin{equation}{section}

\begin{document}

\title{$3\times 3$ Minors of Catalecticants}

\author{Claudiu Raicu}
\address{Department of Mathematics, Princeton University, Princeton, NJ 08544-1000\newline
\indent Institute of Mathematics ``Simion Stoilow" of the Romanian Academy}
\email{craicu@math.princeton.edu}

\subjclass[2000]{Primary 14M12}

\date{\today}

\keywords{Catalecticant matrices, Veronese varieties, secant varieties}

\begin{abstract} Secant varieties of Veronese embeddings of projective space are classical varieties whose equations are far from being understood. Minors of catalecticant matrices furnish some of their equations, and in some situations even generate their ideals. Geramita conjectured that this is the case for the secant line variety of the Veronese variety, namely that its ideal is generated by the $3\times 3$ minors of any of the ``middle'' catalecticants. Part of this conjecture is the statement that the ideals of $3\times 3$ minors are equal for most catalecticants, and this was known to hold set-theoretically. We prove the equality of $3\times 3$ minors and derive Geramita's conjecture as a consequence of previous work by Kanev.
\end{abstract}

\maketitle

\section{Introduction}

A folklore theorem (see \cites{gruson-peskine,eis-detl,conca}) states that the defining ideal of any secant variety of a rational normal curve is generated by the minors of a generic Hankel matrix, and that apart from trivial restrictions, it doesn't matter which Hankel matrix we choose to take the minors of. For example, consider a rational quartic curve $C$ in $\bb{P}^4$, the image of the embedding
\[[x:y]\longrightarrow[x^4:x^3y:x^2y^2:xy^3:y^4].\]
If we let $z_0,\cdots,z_4$ denote the coordinate functions on $\bb{P}^4$, then the relevant Hankel matrices are 
\[ \left[ \begin{array}{cccc}
z_0 & z_1 & z_2 & z_3 \\
z_1 & z_2 & z_3 & z_4 \\
\end{array} \right]\quad
\rm{and}\quad
\left[\begin{array}{ccc}
z_0 & z_1 & z_2 \\
z_1 & z_2 & z_3 \\
z_2 & z_3 & z_4 \\
\end{array} \right].
\]
The ideals of $2\times 2$ minors of the two matrices coincide and generate the ideal of $C$, while the determinant of the second one cuts out the cubic $3$-fold which is the union of the lines secant to $C$. The union of the higher dimensional planes secant to $C$ covers the whole ambient space $\bb{P}^4$, which corresponds to the fact that the above matrices don't have minors of size larger than three.

Unlike the case of $\bb{P}^1$ which is completely understood, we do not know in general the equations of the secant varieties of Veronese embeddings of higher dimensional projective spaces. Minors of catalecticant matrices (which are generalized versions of Hankel matrices, see Section \ref{subcatvar} for definitions) provide some equations for these secant varieties, but turn out not to be sufficient in many cases.

Determinantal loci of catalecticant matrices are of particular interest in their own right, but also via their connection to Hilbert functions of Gorenstein Artin algebras, the polynomial Waring problem, or configurations of points in projective space (see \cites{ger-inversesystems,kanev-iarrobino}). In \cite{geramita}, Geramita gives a beautiful exposition of classical results about catalecticant varieties, and proposes several further questions (see also \cite[Chapter~9]{kanev-iarrobino}). We recall the last one, which we shall answer affirmatively. It is divided into two parts:

\begin{question}\label{que:Geramita} Write $Cat(t,d-t;n)$ for the $t$-th generic catalecticant matrix (see Section~\ref{subcatvar}), and $I_3(Cat(t,d-t;n))$ for the ideal generated by its $3\times 3$ minors.
\begin{enumerate}
\item[(a)]\label{Q5a} Is it true that
\[I_3(Cat(2,d-2;n))=I_3(Cat(t,d-t;n))\]
for all $t$ with $2\leq t\leq d-2$?

\item[(b)]\label{Q5b} Is it true that for $n\geq 3$ and $d\geq 4$
\[I_3(Cat(1,d-1;n))\subsetneq I_3(Cat(2,d-2;n))?\]
\end{enumerate}
\end{question}

Our main result answers positively both parts of Question~\ref{que:Geramita}:

\begin{tmain} Let $K$ be a field of characteristic $0$ and let $n,d\geq 2$ be integers. The following statements hold:
\begin{enumerate}
\item For all $t$ with $2\leq t\leq d-2$ one has
\[I_3(Cat(2,d-2;n))=I_3(Cat(t,d-t;n)).\]

\item If $d\geq 4$ then there is a strict inclusion
\[I_3(Cat(1,d-1;n))\subsetneq I_3(Cat(2,d-2;n)).\]
\end{enumerate}
\end{tmain}

Geramita also conjectures that any of the catalecticant ideals $I_3(Cat(t,d-t;n))$, $2\leq t\leq d-2$, is the ideal of the secant line variety of the $d$-uple embedding of $\bb{P}^{n-1}$. This follows by combining Theorem \ref{thmmain} with the result of Kanev \cite{kanev} which states that the ideal of the secant line variety of the Veronese variety is generated by the $3\times 3$ minors of the first and second catalecticants:

\begin{tmaincor} Any of the ideals $I_3(Cat(t,d-t;n))$, $2\leq t\leq d-2$, is the ideal of the secant line variety of the $d$-th Veronese embedding of $\bb{P}_K^{n-1}$.
\end{tmaincor}

As mentioned earlier, when $n=2$ it is well-known \cites{gruson-peskine,eis-detl,conca} that
\begin{equation}\label{equ:ratlnormalcurve}
I_k(Cat(k-1,d-k+1;2))=I_k(Cat(t,d-t;2))
\end{equation}
for all $t$ with $k-1\leq t\leq d-k+1$, and that any of these ideals is the ideal of the $(k-2)$-nd secant variety of the rational normal curve in $\bb{P}^d$. This fact will turn out to be useful in the proof of Theorem \ref{thmmain}.

Theorem \ref{thmmain} yields special cases of two general conjectures. One of them is implicit in Geramita's question Q4 from \cite{geramita}:
\begin{conjecture}\label{conj:ger} For all $k,n\geq 2$, $d\geq 2k-2$ and $t$ with $k-1\leq t\leq d-k+1$, one has
\[I_k(Cat(k-1,d-k+1;n))=I_k(Cat(t,d-t;n)).\]
Moreover, the following inclusions hold:
\[I_k(Cat(1,d-1;n))\subsetneq I_k(Cat(2,d-2;n))\subsetneq\cdots\subsetneq I_k(Cat(k-1,d-k+1;n)).\]
\end{conjecture}
The other one is a conjecture by Sidman and Smith \cite{SS}:
\begin{conjecture}\label{conjss} Let $k$ be a positive integer. If $X\subset\bb{P}^n$ is embedded by the complete linear series of a sufficiently ample line bundle, then the homogeneous ideal of the $(k-2)$-nd secant variety of $X$ is generated by the $k\times k$-minors of a $1$-generic matrix of linear forms.
\end{conjecture}

Conjecture \ref{conjss} has been proven to be false for singular $X$ \cite{bgl}, but there are no known smooth counterexamples. The case $X=\bb{P}^r$ is a sufficiently interesting special case. In \cite{bucz-bucz} it is shown that minors of catalecticants are not enough to cut out the secant varieties even for very positive embeddings of projective space. Both conjectures~\ref{conj:ger} and~\ref{conjss} are known to be true for $k=2$, by results of Pucci \cite{pucci} and Sidman and Smith \cite{SS}. The argument in \cite{pucci} is rather long, so we will give a simplified proof in Section~\ref{sec2x2}. The case $k=3$ is treated in Section~\ref{secmain}. We prove the case $k=4$ of Conjecture~\ref{conj:ger} in~\cite[Section~6.3]{raicu}, using similar techniques.

The main tool that we will be using in our proofs is the reduction to the ``generic'' situation, as explained in \cite{raiGSS}. Showing the equality of the spaces of minors for the various catalecticants reduces to the more combinatorial problem of showing that certain representations of a symmetric group coincide.

The structure of the paper is as follows. In Section \ref{secpre} we establish some notation from representation theory, and recall some basic facts about catalecticant varieties and secant varieties of Veronese varieties. In particular, we describe the relationship between catalecticant matrices and Gorenstein Artin algebras, which motivates Conjecture \ref{conj:ger}. In Section \ref{secgeneric} we set up the ``generic case'': we introduce certain representations of symmetric groups which correspond by specialization to ideals of minors of catalecticant matrices. We then illustrate our techniques in Section \ref{sec2x2} by giving a simple proof of Pucci's result - Conjecture \ref{conj:ger} in the case $k=2$. In Section \ref{secmain} we give an affirmative answer to Geramita's Question~\ref{que:Geramita} (Theorem \ref{thmmain}).

\section{Preliminaries}\label{secpre}

\subsection{Representation theory}

For an introduction to the representation theory of the symmetric and general linear groups see \cite{ful-har} or \cite{macdonald}. Given a finite dimensional vector space $V$ over a field $K$ of characteristic zero, we denote by $GL(V)$ the group of invertible linear transformations of $V$. For a positive integer $N$, we write $S_N$ for the group of permutations of the set $\{1,\cdots,N\}$.

A \defi{partition} $\ll$ of an integer $N$ is a nonincreasing sequence $\ll_1\geq\ll_2\geq\cdots$ with $N=\sum\ll_i$. We write $\ll=(\ll_1,\ll_2,\cdots)$. Alternatively, if $\mu$ is a partition having $i_j$ parts equal to $\mu_j$ for all $j$, then we write $\mu=(\mu_1^{i_1}\cdot\mu_2^{i_2}\cdots)$. To a partition $\ll$ we associate a \defi{Young diagram} which consists of left-justified rows of boxes, with $\ll_i$ boxes in the $i$-th row. We shall identify a partition $\ll$ with its Young diagram. A \defi{tableau} is a filling of the Young diagram. The \defi{canonical tableau} is the one that numbers the boxes consecutively from left to right, top to bottom. For $\ll=(3,3,1)=(1^1\cdot 3^2)$, the canonical tableau is
\[\Yvcentermath1\young(123,456,7).\]

The irreducible representations  of $GL(V)$ and $S_N$ that will concern us are classified by partitions. For $GL(V)$, they are the \defi{Schur functors} $S_{\ll}V$, with the convention that $S_{(d)}V=\rm{Sym}^d V$ is the $d$-th symmetric power of $V$, while $S_{(1^k)}V=\Lambda^k V$ is the $k$-th exterior power of $V$. For $G=S_N$, we write $[\ll]$ for the irreducible representation corresponding to the partition $\ll$. $[(N)]$ denotes the trivial representation, while $[(1^N)]$ denotes the sign representation.

The $GL(V)$-- (resp. $S_N$--) representations $W$ that we consider decompose as a direct sum of $S_{\ll}V$'s (resp. $[\ll]$'s). We write
\[W=\bigoplus_{\ll}W_{\ll},\]
where $W_{\ll}\simeq (S_{\ll}V)^{m_{\ll}}$ (resp. $W_{\ll}\simeq [\ll]^{m_{\ll}}$) is the \defi{$\ll$-isotypic component} of $W$. $m_{\ll}=m_{\ll}(W)$ is called the \defi{multiplicity} of $S_{\ll}V$ (resp. $[\ll]$) in $W$.

Up to making some choices, each $W_{\ll}$ contains a distinguished subspace $\rm{hwt}_{\ll}(W)$, called the \defi{$\ll$-highest weight space} of $W$. For $GL(V)$, this is the space of vectors of weight $\ll$ invariant under (some choice of) the Borel subgroup, while for $S_N$ it is the vector space $c_{\ll}\cdot W$, where $c_{\ll}$ is a Young symmetrizer. The defining property that will be important for us is that $\rm{hwt}_{\ll}(W)$ is a vector space of dimension $m_{\ll}(W)$ which generates $W_{\ll}$ as a $GL(V)$-- (resp. $S_N$--) representation.

\subsection{Catalecticant varieties}\label{subcatvar}

Given a vector space $V$ of dimension $n$ over $K$, with basis $\mc{B}=\{x_1,\cdots,x_n\}$, we consider its dual space $V^*$ with dual basis $\mc{E}=\{e_1,\cdots,e_n\}$. For every positive integer $d$ we get a basis of $S_{(d)}V^*$ consisting of \defi{divided power monomials} $e^{(\a)}$ of degree $d$ in the $e_i$'s, as follows. If $\a\subset\{1,\cdots,n\}$ is a multiset of size $|\a|=d$, then we write $e^{\a}$ for the monomial 
\[\prod_{i\in\a}e_i.\]
We often identify $\a$ with the multiindex $(\a_1,\cdots,\a_n)$, where $\a_i$ represents the number of occurrences of $i$ in the multiset $\a$. We write $e^{(\a)}$ for $e^{\a}/\a!$, where $\a!=\a_1!\cdots\a_n!$. For $a,b>0$ with $a+b=d$ we get a \defi{divided power multiplication} map $S_{(a)}V^*\o S_{(b)}V^*\to S_{(d)}V^*$, sending $e^{(\a)}\o e^{(\b)}$ to $e^{(\a\cup\b)}$. We can represent this via a multiplication table whose rows and columns are indexed by multisets of sizes $a$ and $b$ respectively, and whose entry in position $(\a,\b)$ is $e^{(\a\cup\b)}$. The \defi{generic catalecticant matrix} $Cat(a,b;n)$ is defined to be the matrix obtained from this multiplication table by replacing each $e^{(\a\cup\b)}$ with the variable $z_{\a\cup\b}$, where $(z_{\c})_{|\c|=d}\subset S_{(d)}V$ is the dual basis to $(e^{(\c)})_{|\c|=d}\subset S_{(d)}V^*$.

One can also think of $z_{\c}$'s as the coefficients of the generic form of degree $d$ in the $e_i$'s, $F=\sum z_{\c}e^{(\c)}$. Specializing the $z_{\c}$'s we get an actual form $f\in S_{(d)}V^*$, and we denote the corresponding \defi{catalecticant matrix} by $Cat_f(a,b;n)$. Any such form $f$ is the \defi{dual socle generator} of some Gorenstein Artin algebra $A$ \cite[Thm.\ 21.6]{eis-CA} with socle degree $d$ and Hilbert function
\[h_i(A)=\rm{rank}(Cat_f(i,d-i;n)).\]
Macaulay's theorem on the growth of the Hilbert function of an Artin algebra \cite[Thm.\ 4.2.10]{bru-her} implies that if $h_i<k$ for some $i\geq k-1$, then the function becomes nonincreasing from that point on. In particular, since $A$ is Gorenstein, $h$ is symmetric, so if $h_i<k$ for some $k-1\leq i\leq d-k+1$, then we have
\[h_1\leq h_2\leq\cdots\leq h_{k-1}=h_k=\cdots=h_{d-k+1}\geq h_{d-k+2}\geq\cdots\geq h_{d}.\]

If we denote by $I_k(i)=I_k(Cat(i,d-i;n))$ the ideal of $k\times k$ minors of the $i$-th generic catalecticant, then the remarks above show that whenever $k-1\leq d-k+1$ we have the following up-to-radical relations:
\[I_k(1)\subset\cdots\subset I_k(k-1)=\cdots=I_k(d-k+1)\supset\cdots\supset I_k(d-1).\]
Conjecture \ref{conj:ger} states that these relations hold exactly. We prove the conjecture in the case $k=3$ in Section \ref{secmain}.

\subsection{Secant Varieties of Veronese Varieties}\label{seceqnsec}

Given a vector space $U$ over a field $K$ of characteristic zero, we write $\bb{P}U$ for the projective space of lines in $U$. For $0\neq u\in U$, we denote by $[u]$ the corresponding line. For $d$ a positive integer, we consider the Veronese embedding
\[\rm{Ver}_d:\bb{P}(V^*)\to\bb{P}(S_{(d)}V^*),\ \textrm{ given by }\ [e]\mapsto[e^{(d)}].\]
Its \defi{$k$-th secant variety} is the Zariski closure of the union of the linear subspaces spanned by collections of $k+1$ points in the image of $\rm{Ver}_d$. We denote it by $\s_{k+1}(\rm{Ver}_d(\bb{P}V^*))$. Note that for $k=0$ this is just the image of $\rm{Ver}_d$.

The homogeneous coordinate ring of $\bb{P}(S_{(d)}V^*)$ is $S=\rm{Sym}(S_{(d)}V)$, the symmetric algebra over $S_{(d)}V$. Using the basis $(z_{\a})\subset S_{(d)}V$ dual to $(e^{(\a)})\subset S_{(d)}V^*$ we can write $S$ as the polynomial ring $K[z_{\a}]$. An important open problem is to find the ideal $I\subset S$ of polynomials vanishing on $\s_{k}(\rm{Ver}_d(\bb{P}V^*))$ (see \cite{LO} for the current state of the art). The following result is well-known (see \cite{kanev-iarrobino} or \cite{landsberg}).

\begin{lemma}\label{lem:eqnsminors} For every $1\leq i\leq d$ and $k\geq 1$, the ideal $I_{k+1}(Cat(i,d-i;n))$ is contained in the ideal of $\s_{k}(\rm{Ver}_d(\bb{P}V^*))$.
\end{lemma}

\section{The ``Generic'' Case}\label{secgeneric}

The material in this section is based on~\cite[Section~3B]{raiGSS}. The basic idea is to use Schur-Weyl duality in order to translate questions about representations of general linear groups, such as the vector spaces spanned by the minors of catalecticants, into questions about representations of symmetric groups and tableau combinatorics.

We write $N=N(r,d)=r\cdot d$, and consider the vector space $W_d^r$ with basis consisting of monomials in commuting variables $z_{\a_i}$
\[
\begin{split}
m=z_{\a_1}\cdots z_{\a_r},\ \rm{ where }\ \a_1\sqcup\cdots\sqcup\a_r &\ \textrm{ is a partition of the set }\ \{1,\cdots,N\},\\
\ \rm{ with }\ |\a_i|=d &\ \textrm{ for all }\ i=1,\cdots,r.
\end{split}
\]
An element $\s$ of the symmetric group $S_N$ acts on a monomial $m$ as follows:
\[\s(m)=\s(z_{\a_1}\cdots z_{\a_r})=z_{\s(\a_1)}\cdots z_{\s(\a_r)},\]
where for a subset $\a\subset\{1,\cdots,N\}$, $\s(\a)=\{\s(x):x\in \a\}$.

\begin{definition}[Generic flattenings]\label{def:genericflattenings}
For $k\leq r$, $a,b$ with $a+b=d$, and disjoint subsets $\a_1,\cdots,\a_k,\b_1,\cdots,\b_k\subset\{1,\cdots,N\}$ with $|\a_i|=a$, $|\b_i|=b$ for all $i=1,\cdots k$, we let
\[[\a_1,\cdots,\a_k|\b_1,\cdots,\b_k]=\rm{det}(z_{\a_i\cup\b_j})_{1\leq i,j\leq k}.\]
Fixing $k,d$ and $a,b$ with $a+b=d$, we define the \defi{ideal of generic $k\times k$ minors of the $a$-th catalecticant} to be the collection, indexed by the degree $r$, of subrepresentations $I_k^r(a,b)\subset W_d^r$ spanned by the expressions
\[[\a_1,\cdots,\a_k|\b_1,\cdots,\b_k]\cdot z_{\c_1}\cdots z_{\c_{r-k}},\]
where $\a_1\sqcup\cdots\sqcup\a_k\sqcup\b_1\sqcup\cdots\sqcup\b_k\sqcup\c_1\sqcup\cdots\sqcup\c_{r-k}$ form a partition of the set $\{1,\cdots,N\}$, with $|\a_i|=a$, $|\b_i|=b$, $|\c_i|=d$. %When $r$ is understood from the context, we write $I_k(a,b)$ for the representation $I_k^r(a,b)$.
\end{definition}

\begin{example}\label{ex:genericflattenings} Take $d=k=r=3$, $a=2$ and $b=1$. A typical element of $I_3^3(a,b)$ looks like
\[D=[\{1,2\},\{4,6\},\{5,8\}|\{3\},\{7\},\{9\}]=\rm{det}\left[ \begin{array}{ccc}
z_{\{1,2,3\}} & z_{\{1,2,7\}} & z_{\{1,2,9\}} \\
z_{\{4,6,3\}} & z_{\{4,6,7\}} & z_{\{4,6,9\}} \\
z_{\{5,8,3\}} & z_{\{5,8,7\}} & z_{\{5,8,9\}} \\
\end{array} \right].\]
\end{example}

It would be desirable to understand the decomposition into irreducible $S_N$--representations of all $I_k^r(a,b)$. This is of course a hopeless goal at this point, since not even the case $k=1$, i.e. the symmetric plethysm problem of decomposing $W_d^r$, is understood in general. Nevertheless, we will be able to achieve our goal in the case of the representations $I_2^2(a,b)$ and $I_3^3(a,b)$. This will allow us to prove conjectures \ref{conj:ger} and \ref{conjss} in the special case $k=3$, $X=\bb{P}^n$, and to reprove Pucci's result (Theorem \ref{thmpuc}). We start with a general observation:

\begin{proposition}\label{prop:1flattenings} For any $k,r,d$, the subrepresentation $I_k^r(1,d-1)\subset W_d^r$ is the sum of the irreducible subrepresentations of $W_d^r$ corresponding to partitions $\ll$ with at least $k$ parts.
\end{proposition}
\begin{proof} This is a special case of Proposition 3.20 in \cite{raiGSS}.
\end{proof}

\begin{remark} Proposition \ref{prop:1flattenings} is the analogue in the setting of $S_N$-representations of Corollary 7.2.3 in \cite{weyman} or Theorem 5.2.3.6 in \cite{landsberg}. The representations $I_k^r(1,d-1)$ give the ``generic equations'' for the subspace varieties.
\end{remark}

Given a partition $\ll$ of $N$, we index the boxes of its Young diagram in the usual way: the $i$-th box is the one whose entry in the canonical tableau is equal to $i$. Given a partition $\ll$ and monomial $m=z_{\a_1}\cdots z_{\a_r}$, we identify the element $c_{\ll}\cdot m\in\rm{hwt}_{\ll}(W_d^r)$ with a tableau of shape $\ll$, having $d$ entries equal to $i$ in the positions indexed by the elements of the set $\a_i$. For example, if $\ll=(6,2)$, $r=d=3$, $m=z_{1,3,8}\cdot z_{2,4,7}\cdot z_{5,6,9}$, we write
\[\Yvcentermath1 c_{\ll}\cdot m=\young(121233,213).\]
Two tableaux differing by a permutation of the numbers $\{1,\cdots,r\}$ correspond to the same monomial, so we identify them:
\[\Yvcentermath1 c_{\ll}\cdot m=c_{\ll}\cdot z_{2,4,7}\cdot z_{5,6,9}\cdot z_{1,3,8}=\young(313122,132).\]

\begin{lemma}\label{lem:relstableaux} With the above conventions, we have
\begin{enumerate}
\item $T = c_{\ll}\cdot m = 0$ if $T$ has repeated entries in some column.
\item $T_1 = c_{\ll}\cdot m_1$ and $T_2 = c_{\ll}\cdot m_2$ are equal up to sign ($T_1=\pm T_2$) if $T_1,T_2$ differ by permutations within columns or by permutations of columns of the same size.
\end{enumerate}
\end{lemma}

\begin{proof} This is a special case of Lemma 3.16 in \cite{raiGSS}.
\end{proof}

\begin{definition}\label{def:circledtableaux}
 Let $a,b$ and $I_k^r(a,b)$ as in Definition \ref{def:genericflattenings}, let
\[D = [\a^1,\cdots,\a^k|\b^1,\cdots,\b^k]\cdot z_{\c^{k+1}}\cdots z_{\c^r}\in I_k^r(a,b),\]
and let $\ll$ be a partition of $N$. We let $\c^i=\a^i\cup\b^i$ for $i=1,\cdots,k$, and consider the tableau $T=c_{\ll}\cdot m$ corresponding to the monomial $m=z_{\c^1}\cdots z_{\c^r}$. We represent the element $c_{\ll}\cdot D\in\rm{hwt}_{\ll}(I_k^r(a,b))$ as a tableau $\hat{T}$ with some of the entries circled, obtained from $T$ by circling the entries in the boxes located at positions indexed by the elements of $\a^1,\cdots,\a^k$. Alternatively, since $[\a^1,\cdots,\a^k|\b^1,\cdots,\b^k]=[\b^1,\cdots,\b^k|\a^1,\cdots,\a^k]$, we can circle the entries in the boxes located at positions indexed by the elements of $\b^1,\cdots,\b^k$, and obtain a different tableau with circled entries which represents the same element of $I_k^r(a,b)$. Using the expansion of the determinant
\[\hat{T}=\sum_{\s\in S_k}\rm{sgn}(\s)\cdot T^{\s},\]
where $\rm{sgn}(\s)$ denotes the sign of $\s$, and $T^{\s}$ is the tableau obtained from $T$ by replacing the entries located at positions indexed by elements of $\a^i$ with $\s(i)$, for $i=1,\cdots,k$.
\end{definition}

\begin{example}\label{ex:circledtableaux}
 Let $D$ be as in Example \ref{ex:genericflattenings}, $\a^1=\{1,2\}$, $\a^2=\{4,6\}$, $\a^3=\{5,8\}$, $\b^1=\{3\}$, $\b^2=\{7\}$, and $\b^3=\{9\}$. We get
\[m=z_{1,2,3}\cdot z_{4,6,7}\cdot z_{5,8,9},\]
so that
\[\Yvcentermath1 T = c_{\ll}\cdot m = \young(11123,223,3),\]
and
\[\Yvcentermath1 \hat{T} = c_{\ll}\cdot D = \young(\ccone\ccone1\cctwo\ccthree,\cctwo2\ccthree,3).\]
We have
\[\Yvcentermath1 \hat{T} = T - \young(22113,123,3) - \young(33121,221,3) - \young(11132,322,3) + \young(22131,321,3) + \young(33112,122,3).\]
Notice that all the tableaux pictured above have a repeated entry in one of their columns, hence are equal to zero by Lemma \ref{lem:relstableaux}. This shows that $T=\hat{T}\in I_3^3(2,1)$. This example captures the main trick in our proof of Geramita's conjecture.
\end{example}

\begin{proposition}\label{prop:genericspecial}
Let $W$ denote the $GL(V)$-representation $S_{(r)}(S_{(d)}V)$, let $N=r\cdot d$ and let $W'=W_d^r$ be the $S_N$-representation described above. We fix a partition $\ll$ of $N$ having at most $n=\rm{dim}(V)$ parts. There exist polarization and specialization maps 
\[P_{\ll}:\rm{hwt}_{\ll}(W)\lra\rm{hwt}_{\ll}(W'),\quad Q_{\ll}:\rm{hwt}_{\ll}(W')\lra\rm{hwt}_{\ll}(W),\]
inducing inverse isomorphisms between
\begin{equation}\label{eq:hwtI}
\rm{hwt}_{\ll}(I_k^r(a,b))\simeq\rm{hwt}_{\ll}(I_k(Cat(a,b;n))_r).
\end{equation}
\end{proposition}

\begin{proof}
 This is a special case of Proposition 3.27 in \cite{raiGSS}.
\end{proof}

It follows that in order to show that $I_k(Cat(a,b;n))$ are all the same as long as $a,b\geq 2$, it suffices to prove the corresponding statement in the generic case, i.e. for the representations $I_k^r(a,b)$.

\begin{corollary}[Inheritance, \cite{landsberg}]\label{cor:inheritance}
 Let $k,r\geq 0$ and fix $\ll$ a partition with $t$ parts. The multiplicity of the irreducible representation $S_{\ll}V$ in $I_k(Cat(a,b;n))_r$ is independent of the dimension $n$ of the vector space $V$, as long as $t\leq n$.
\end{corollary}

\begin{proof}
 The left hand side of (\ref{eq:hwtI}) is independent on $n$, and the isomorphism holds as long as $t\leq n$.
\end{proof}

\section{$2\times 2$ Minors}\label{sec2x2}

In this section we give two proofs of the following result of Pucci, which is the case $k=2$ of Conjecture \ref{conj:ger}. The first proof works in arbitrary characteristic, while the second one is a characteristic zero proof meant to illustrate the methods that we shall use in the case of higher minors.

\begin{theorem}[\cite{pucci}]\label{thmpuc} Let $K$ be a field of arbitrary characteristic and let $n,d\geq 2$ be integers. For all $t$ with $1\leq t\leq d-1$ we have
\[I_2(Cat(1,d-1;n))=I_2(Cat(t,d-t;n)).\]
\end{theorem}

\begin{proof}[Proof in arbitrary characteristic] For multisets $m_1,m_2,n_1,n_2$ we let
\[[m_1,m_2|n_1,n_2]=\left|\begin{array}{ccc}
	z_{m_1\cup n_1} & z_{m_1\cup n_2} \\
	z_{m_2\cup n_1} & z_{m_2\cup n_2} \\
\end{array}\right|.
\]
With this notation, we have the following identity for multisets $u_1$, $u_2$, $v_1$, $v_2$, $\a_1$, $\a_2$, $\b_1$, $\b_2$:
\begin{equation}\label{rel2x2}
\begin{split}
[u_1\cup u_2,v_1\cup v_2|\a_1\cup \a_2,\b_1\cup \b_2]&=[u_1\cup \a_1,v_1\cup \b_1|u_2\cup \a_2,v_2\cup \b_2]\\
&+[u_1\cup \b_2,v_1\cup \a_2|v_2\cup \a_1,u_2\cup \b_1].
\end{split}
\end{equation}

We shall prove that $I_2(Cat(a,b;n))\subset I_2(Cat(a+1,b-1;n))$ for $a+b=d$ and $1\leq a\leq d-2$. This is enough to prove the equality of the $2\times 2$ minors of all the catalecticants, since $I_2(Cat(1,d-1;n))=I_2(Cat(d-1,1;n))$. Since the ideal $I_2(Cat(a,b;n))$ is generated by minors $[m_1,m_2|n_1,n_2]$ with $|m_1|=|m_2|=a$ and $|n_1|=|n_2|=b$, it follows from \ref{rel2x2} that it's enough to decompose $m_1,m_2,n_1,n_2$ as
\[m_1=u_1\cup u_2,\quad m_2=v_1\cup v_2,\quad n_1=\a_1\cup \a_2,\quad n_2=\b_1\cup \b_2,\]
in such a way that
\begin{equation}\label{2x2v1}
\begin{split}
|u_1|+|\a_1|=|v_1|+|\b_1|=a+1,&\quad |u_2|+|\a_2|=|v_2|+|\b_2|=b-1,\\
|u_1|+|\b_2|=|v_1|+|\a_2|=b-1,&\quad |v_2|+|\a_1|=|u_2|+|\b_1|=a+1,
\end{split}
\end{equation}
or
\begin{equation}\label{2x2v2}
\begin{split}
|u_1|+|\a_1|=|v_1|+|\b_1|=a+1,&\quad |u_2|+|\a_2|=|v_2|+|\b_2|=b-1,\\
|u_1|+|\b_2|=|v_1|+|\a_2|=a+1,&\quad |v_2|+|\a_1|=|u_2|+|\b_1|=b-1.
\end{split}
\end{equation}

If $a\leq 2b-2$, then we can find $0\leq x,y\leq b-1$ with $x+y=a$. Choose any such $x,y$ and decompose
\[m_1=u_1\cup u_2,\quad m_2=v_1\cup v_2,\quad\rm{with }\ |u_2|=|v_1|=x\ \rm{ and }\ |u_1|=|v_2|=y,\]
and
\[
\begin{split}
n_1=\a_1\cup \a_2,&\quad n_2=\b_1\cup \b_2,\quad\rm{with }\\
|\a_1|=x+1,\ |\b_1|=y+1,&\ |\a_2|=b-1-x\ \rm{ and }\ |\b_2|=b-1-y.
\end{split}
\]
It's easy to see then that (\ref{2x2v1}) is satisfied.

If $b\leq 2a+2$, then since $b\geq 2$ ($a\leq d-2$), we can find $1\leq x,y\leq a+1$ with $x+y=b$. Choose any such $x,y$ and decompose
\[n_1=\a_1\cup \a_2,\quad n_2=\b_1\cup \b_2,\quad\rm{with }\ |\a_2|=|\b_1|=x\ \rm{ and }\ |\a_1|=|\b_2|=y,\]
and
\[
\begin{split}
m_1=u_1\cup u_2,&\quad m_2=v_1\cup v_2,\quad\rm{with }\\
|u_1|=a+1-y,\ |v_1|=a+1-x,&\ |u_2|=y-1\ \rm{ and }\ |v_2|=x-1.
\end{split}
\]
It's easy to see then that (\ref{2x2v2}) is satisfied.

If neither of $a\leq 2b-2$ and $b\leq 2a+2$ holds, then
\[a\geq 2b-1\geq 2(2a+3)-1=4a+5,\]
so $0\geq 3a+5$, a contradiction.
\end{proof}

\begin{proof}[Proof in characteristic zero] By Proposition \ref{prop:genericspecial}, it's enough to treat the ``generic case''. We want to show that for positive integers $a,b$ with $a+b=d$, and $N=2d$, all $S_N$-subrepresentations $I_2^2(a,b)\subset W_d^2=\rm{ind}_{S_d\sim S_2}^{S_N}(\bf{1})$ are the same. Clearly the trivial representation $[(N)]$ is not contained in any $I_2^2(a,b)$, so
\[I_2^2(a,b)\subseteq W_d^2/[(N)]=\bigoplus_{i=1}^{\lfloor d/2\rfloor}[(2\cdot(d-i),2\cdot i)],\ \textrm{ for all }\ a,b\ \rm{ with }\ a+b=d.\]
(see \cite[I.8,~Ex.~6]{macdonald} for the formula of the decomposition of $W_d^2$ into irreducible representations; as the rest of the proof will show, we don't really need the precise description of this decomposition).

We will finish the proof by showing that the above inclusions are in fact equalities for all $a,b$. To see this, it's enough to prove that for any $a,b$ with $a+b=d$, any partition $\ll$ with two parts, and any monomial $m=z_{\a}\cdot z_{\b}$, with $\a\sqcup\b=\{1,\cdots,N\}$, we have $c_{\ll}\cdot m\in I_2^2(a,b)$. Fix then such $a,b,\ll=(\ll_1,\ll_2)$ and $m=z_{\a}\cdot z_{\b}$. 

Recall from Section \ref{secgeneric} that we can identify $c_{\ll}\cdot m$ with a tableau $T$ of shape $\ll$ with $1$'s in the positions indexed by the elements of $\a$, and $2$'s in the positions indexed by the elements of $\b$. Recall also that if $T$ has repeated entries in a column, then $T = 0$. Since permutations within columns of $T$ can only change the sign of $T$, and permutations of the columns of $T$ of the same size don't change the value of $T$ (Lemma \ref{lem:relstableaux}), we can assume that in fact $m=z_{\{1,\cdots,d\}}\cdot z_{\{d+1,\cdots,N\}}$ and
\[\Yvcentermath1 T=c_{\ll}\cdot z_{\{1,\cdots,d\}}\cdot z_{\{d+1,\cdots,N\}}=\young(111\cdots22\cdots,22\cdots).\]
Consider the sets
\[\a_1=\{2,\cdots,a+1\},\ \a_2=\{1,\cdots,d\}\setminus\a_1,\ \b_1\ \rm{ and }\ \b_2=\{d+1,\cdots,N\}\setminus\b_1,\]
where $\b_1$ is any subset with $a$ elements of $\{d+1,\cdots,N\}$ containing $\ll_1+1$. Let $\tilde{T}$ be the tableau obtained from $T$ by circling the boxes corresponding to the entries of $\a_1$ and $\b_1$ (see Definition~\ref{def:circledtableaux}). We have
\[\tilde{T}=c_{\ll}\cdot[\a_1,\b_1|\a_2,\b_2]=c_{\ll}\cdot(z_{\a_1\cup\a_2}\cdot z_{\b_1\cup\b_2}-z_{\a_1\cup\b_2}\cdot z_{\a_2\cup\b_1})=T-T',\]
where $T'$ is a tableau with two equal entries in its first column, i.e. $T'=0$. We get 
\[T=\tilde{T}\in I_2^2(a,b),\]
completing the proof.
\end{proof}

\begin{remark} The characteristic zero case also follows by inheritance (Proposition~\ref{prop:genericspecial} and Corollary~\ref{cor:inheritance}): since all the partitions $\ll$ that show up have at most two parts, it suffices by inheritance to prove the theorem when $n=2$, but in this case all the catalecticant ideals are the same, as remarked in the introduction (\ref{equ:ratlnormalcurve}).
\end{remark}

\section{Geramita's Conjecture}\label{secmain}

We are now ready to give an affirmative answer to Geramita's Question~\ref{que:Geramita} in the Introduction.

\begin{theorem}\label{thmmain} Let $K$ be a field of characteristic $0$ and let $n\geq 2$, $d\geq 4$ be integers. The following statements hold:
\begin{enumerate}
\item For all $t$ with $2\leq t\leq d-2$ we have
\[I_3(Cat(2,d-2;n))=I_3(Cat(t,d-t;n)).\]

\item There is a strict inclusion
\[I_3(Cat(1,d-1;n))\subsetneq I_3(Cat(2,d-2;n)).\]

\end{enumerate}
\end{theorem}

\begin{corollary}\label{thmmaincor} Any of the ideals $I_3(Cat(t,d-t;n))$, $2\leq t\leq d-2$, is the ideal of the secant line variety of the $d$-th Veronese embedding of $\bb{P}_K^{n-1}$.
\end{corollary}
\begin{proof} This follows from \cite[Theorem 3.3(ii)]{kanev} and Theorem \ref{thmmain}. 
\end{proof}

\begin{proof}[Proof of Theorem \ref{thmmain}] To prove (1), it suffices by Proposition \ref{prop:genericspecial} to show that $I_3^3(2,d-2)=I_3^3(t,d-t)\subset W_d^3$ for $2\leq t\leq d-2$. The $\ll$-highest weight spaces of all $I_3^3(t,d-t)$, $2\leq t\leq d-2$, are the same when $\ll$ has at most two parts. This follows by inheritance: combine Proposition \ref{prop:genericspecial} with the fact that the theorem is known when $n=2$ (\ref{equ:ratlnormalcurve}). We shall prove that when $\ll$ has three parts, the $\ll$-isotypic component of $I_3^3(t,d-t)$ is equal to the $\ll$-isotypic component of $W_d^3$ for all $t$ with $1\leq t\leq d-1$ (we already know this when $t=1$, by Proposition \ref{prop:1flattenings}). This will imply (1) and the inclusion of (2). The reason why this inclusion is strict for $d\geq 4$ is because it is already strict when $n=2$, and because inheritance holds for catalecticant ideals.

Consider a partition $\ll=(\ll_1,\ll_2,\ll_3)$ with $3$ parts, a monomial $m\in W_d^3$ with corresponding tableau $T=c_{\ll}\cdot m$, and integers $2\leq a\leq b$ with $a+b=d$. We shall prove that $T\in I_3^3(a,b)$. We will see that if $\ll$ has only one entry in the second column, then $T=0$, so let's assume this isn't the case for the moment. We may also assume that $T$ has no repeated entries in a column (Lemma~\ref{lem:relstableaux}). Since permuting the numbers $1,2,3$ in the tableau $T$ doesn't change $T$, and permutations within the columns of $T$ preserve $T$ up to sign, we may assume that $T$ contains the subtableau
\[\young(11,22,3)\]
in its first two columns (there may or may not be a third box in the second column of $\ll$).

It follows that $m=z_{\c_1}z_{\c_2}z_{\c_3}$, with $\c_1=\{1,2,\cdots\}$, $\c_2=\{\ll_1+1,\ll_1+2,\cdots\}$ and $\c_3=\{\ll_1+\ll_2+1,\cdots\}$, $|\c_i|=d$. Consider subsets $\a_i\subset\c_i$, $|\a_i|=a$ satisfying the conditions
\[1,2\in\a_1,\quad\ll_1+1\in\a_2,\ll_1+2\notin\a_2,\quad\ll_1+\ll_2+1\notin\a_3,\]
and let $\b_i=\c_i\setminus\a_i$, for $i=1,2,3$. Let $\tilde{T}$ be the tableau obtained from $T$ by circling the entries of $\a_1,\a_2,\a_3$, so that $\tilde{T}\in I_3^3(a,b)$. $\tilde{T}$ looks like
\[\Yvcentermath1 \tilde{T}=\young(\ccone\ccone\cdots,\cctwo2\cdots,3\cdots).\]
We get
\[\tilde{T}=c_{\ll}\cdot [\a_1,\a_2,\a_3|\b_1,\b_2,\b_3]=T+\sum_{j=1}^5\pm T_j,\]
where each $T_j$ is a tableau with repeated entries in one of its first two columns (i.e. $T_j=0$). It follows that
\[T = \tilde{T}\in I_3^3(a,b),\]
as in Example \ref{ex:circledtableaux}, which is what we wanted to prove.

To see that $T=c_{\ll}\cdot m=0$ for all monomials $m$ when $\ll=(3d-2,1,1)$, it suffices to notice that if $\s$ is the transposition of the $(3d-1)$-st and $3d$-th boxes of $\ll$ (the $2$nd and $3$rd boxes in the first column of $\ll$), then $\s(T)=c_{\ll}\cdot(\s\cdot m)$ and $T$ are the same up to permutations of columns size $1$ (and permutations of the entries of the alphabet $\mc{A}=\{1,2,3\}$). It follows that
\[c_{\ll}\cdot m = c_{\ll}\cdot(\s\cdot m) = (c_{\ll}\cdot\s)\cdot m=-c_{\ll}\cdot m,\]
so that $T=c_{\ll}\cdot m=0$, as desired. Alternatively, it follows from \cite[I.8,~Ex.~9]{macdonald} that $[\ll]$ does not appear in the decomposition of $W_d^3$ into irreducible representations, so $c_{\ll}\cdot W_d^3=0$.
\end{proof}

\section*{Acknowledgments} 
I would like to thank David Eisenbud for his guidance throughout the project, Charley Crissman for numerous conversations on the subject, and Tony Geramita, Tony Iarrobino and Bernd Sturmfels for helpful comments and suggestions. I also thank Dan Grayson and Mike Stillman for making Macaulay2 \cite{M2}, which served as a good source of inspiration. This work was partially supported by National Science Foundation Grant No. 0964128.

	%%%%%%%%%%%%%%%%%%%%%%%%%%%%%%%%%%%%%%%%%%%%%%%%%%%%%%%%%%%%%%%%%%%%%%%%
	%%%%%%%%%%%%%%%   		Bibliography				%%%%%%%%%%%%%%%%%%%%
	%%%%%%%%%%%%%%%%%%%%%%%%%%%%%%%%%%%%%%%%%%%%%%%%%%%%%%%%%%%%%%%%%%%%%%%%

\end{document}